\documentclass[12pt,a4paper,oneside,reqno]{amsart}

\usepackage{psfrag,graphicx}
\usepackage{geometry}
\usepackage{url}

\newcommand{\tree}[1]{\mathrm{Tree}(#1)}
\newcommand{\utree}[1]{\mathrm{UTree}(#1)}
\newcommand{\downward}[2]{\mathcal{D}_{#1}(#2)}
\newcommand{\downwardi}[1]{\downward{i}{#1}}
\newcommand{\dir}[1]{\mathrm{dir}(#1)}
\newcommand{\dd}[1]{\mathrm{dd}(#1)}
\newcommand{\sgn}[1]{\mathrm{sgn}(#1)}

\newcommand{\islide}[2][i]{\ensuremath{\mathcal{E}_{#1}(#2)}}

\newcommand{\symd}{\ominus}

\newcommand{\Fp}[1]{\ensuremath{F_{#1}^+}}
\newcommand{\Fm}[1]{\ensuremath{F_{#1}^-}}
\newcommand{\eslide}[1]{\ensuremath{\mathcal{E}(Q_{#1})}}
\newcommand{\eslidesig}[1]{\ensuremath{\mathcal{E}(#1)}}

\newcommand{\stacking}[2]{\genfrac{}{}{0pt}{}{#1}{#2}}
\newcommand{\power}[1]{\mathcal{P}(#1)}
\newcommand{\powernon}[1]{\mathcal{P}_{\geq 1}^{#1}}
\newcommand{\ptwo}[1]{\mathcal{P}_{\geq 2}^{#1}}

\newcommand{\eps}{\varepsilon}
\def\co{\colon\thinspace}

\newtheorem{theorem}{Theorem}
\newtheorem{lemma}[theorem]{Lemma}
\newtheorem{corollary}[theorem]{Corollary}

\theoremstyle{remark}
\newtheorem{remark}[theorem]{Remark}
\newtheorem*{example}{Example}

\begin{document}

\title{Counting the spanning trees of the $3$-cube using edge slides}
\author{Christopher Tuffley}
\address{Institute of Fundamental Sciences, Massey University,
         Private Bag 11 222, Palmerston North 4442, New Zealand}
\email{c.tuffley@massey.ac.nz}


\begin{abstract}
We give a direct combinatorial proof of the known fact that the 3-cube
has 384 spanning trees, using an ``edge slide'' operation on spanning
trees. This gives an answer in the case $n=3$ to a question implicitly
raised by Stanley. Our argument also gives a bijective proof of the
$n=3$ case of a weighted count of the spanning trees of the $n$-cube
due to Martin and Reiner.
\end{abstract}

\maketitle
\renewcommand{\thefootnote}{}
\footnotetext{\emph{Journal reference:} \emph{Australas.\ J.\ Combin.}, 54:189--206, 2012.}

\section{Introduction}
 
The $n$-cube is the graph $Q_n$ whose vertices are the
subsets of the set $[n]=\{1,2,\ldots,n\}$, with an edge between $S$ and
$R$ if they differ by the addition or removal of precisely one element.
The $3$-cube is then the familiar graph shown in Figure~\ref{qthree.fig},
whose edges and vertices form the edges and vertices of an ordinary 
cube or die.

The number of spanning trees of the $n$-cube is known via Kirchhoff's
Matrix-Tree Theorem (see for example Stanley~\cite{stanley-II})
to be
\begin{equation}
|\tree{Q_n}| = 2^{2^n-n-1} \prod_{k=1}^n k^{\binom{n}{k}} 
             = \prod_{\stacking{S\subseteq [n]}{|S|\geq 2}} 2|S|.
\label{noftrees.eq}
\end{equation}
However, according to Stanley~\cite[p.~62]{stanley-II} a direct
combinatorial proof of this formula is not known. In contrast, the
spanning trees of the complete graph $K_n$ may be counted not only via
the Matrix-Tree Theorem, but also bijectively, using the Pr\"ufer
Code~\cite{prufer18}. Indeed, there are several known proofs that
$K_n$ has $n^{n-2}$ spanning trees --- see for example
Moon~\cite{moon67}.

The purpose of this note is to give a direct combinatorial proof that
the $3$-cube has $2^4\cdot2^3\cdot3=384$ spanning trees, and thereby
answer in the case $n=3$ the question implicitly raised by Stanley's
comment.  
We will do this by defining and using ``edge slide'' moves
on the spanning trees of $Q_3$ to break the set of trees into
families that are readily counted. Our methods do not readily extend
to $n\geq 4$, but 
since this paper was written, Stanley's question has been answered in
full using different methods by Bernardi~\cite{bernardi2012}.

Our approach is motivated by the following refinement of~\eqref{noftrees.eq},
due to Martin and Reiner~\cite[Thm.~3]{martin-reiner}. Again using the 
Matrix-Tree Theorem, they prove a weighted count of the spanning trees
of $Q_n$ in terms of variables $q_1,\ldots,q_n$ and $x_1,\ldots,x_n$,
namely
\begin{equation}
\label{weightedcount.eq}
\sum_{T\in\tree{Q_n}} q^{\dir{T}}x^{\dd{T}}
    =q_1\cdots q_n \prod_{\stacking{S\subseteq [n]}{|S|\geq 2}}
                \sum_{i\in S} q_i (x_i^{-1}+x_i)
\end{equation}
(see Section~\ref{dddm.sec} for details).
Each term on the right-hand side
corresponds to a tree, and is obtained by choosing, for each
subset $S$ of $[n]$ of size at least $2$, an element $i$ of $S$ and a
sign $\pm 1$. 
We call such a series of choices a \emph{signed section} of
$\ptwo{n}=\{S\in\power{[n]}:|S|\geq2\}$. Our combinatorial proof
of~\eqref{noftrees.eq} for $n=3$ may be used to construct a
weight preserving bijection between the spanning trees of $Q_3$
and the signed sections of $\ptwo{3}$, giving a combinatorial
proof of the $n=3$ case of~\eqref{weightedcount.eq}.

\begin{figure}
\begin{center}
\psfrag{empty}{$\emptyset$}
\psfrag{1}{$\{1\}$}
\psfrag{2}{$\{2\}$}
\psfrag{3}{$\{3\}$}
\psfrag{12}{$\{1,2\}$}
\psfrag{23}{$\{2,3\}$}
\psfrag{13}{$\{1,3\}$}
\psfrag{123}{$\{1,2,3\}$}
\includegraphics[scale=1]{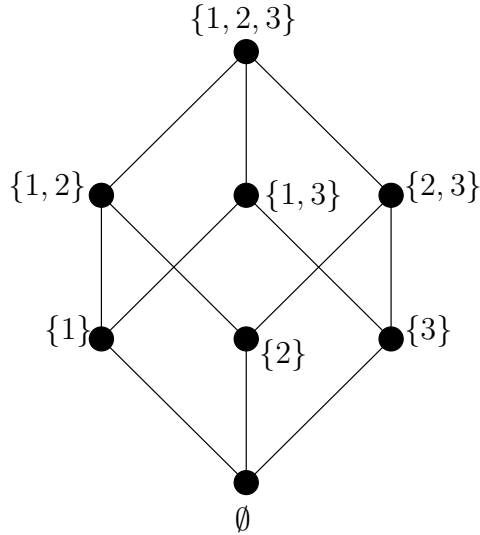}
\caption{The three-cube.}
\label{qthree.fig}
\end{center}
\end{figure}

\section{Martin and Reiner's weighted count}

Before proceeding we describe the weights used in Martin and Reiner's
formula~\eqref{weightedcount.eq}. 
The weight of a tree consists of two factors, the 
\emph{direction monomial} $q^{\dir{T}}$ and the 
\emph{decoupled degree monomial} $x^{\dd{T}}$, defined below
in Section~\ref{dddm.sec}.
In Section~\ref{ddalt.sec} we prove an
alternate formulation of the decoupled degree monomial,
which we will use in what follows.

\subsection{The direction and decoupled degree monomials}
\label{dddm.sec}

Recall that we regard $Q_n$ as the graph with vertex set the power set 
of $[n]$, with an edge between subsets $S$ and $R$
if they differ by the addition or deletion of a single element.
If $S$ and $R$ differ by the addition or deletion of $i\in[n]$
we say that the edge $e=(S,R)$ is \emph{in the direction $i$},
and write $\dir{e}=i$. 
As usual, for a graph $G$ we write $E(G)$ for the edge set of $G$.
With this notation, the direction monomial
of a spanning tree $T$ of $Q_n$ is
\[
q^{\dir{T}} = \prod_{e\in E(T)} q_{\dir{e}}.
\]
To define the decoupled degree monomial, given $S\in[n]$
let $x_S = \prod_{i\in S} x_i$. Then
\begin{align}
x^{\dd{T}} &= \prod_{S\subseteq[n]}
  \left( \frac{x_S}{x_{[n]\setminus S}}\right)^{\frac{1}{2}\deg_T (S)} 
  \nonumber \\
           &= \prod_{(S,R)\in E(T)}  \frac{x_S x_R}{x_{[n]}}.
               \label{dd1.eq}
\end{align}

\subsection{An alternate formulation of the decoupled degree monomial}
\label{ddalt.sec}

Note that if the edge $(S,R)$ of $T$ is in the direction $i$
then
\begin{equation}
\frac{x_S x_R}{x_{[n]}} = x_1^{\epsilon_1}\cdots x_{i-1}^{\epsilon_{i-1}}
                         x_{i+1}^{\epsilon_{i+1}}\cdots x_n^{\epsilon_n},
\label{exhibitA}
\end{equation}
where for $j\not=i$,
\begin{equation}
\epsilon_j = \begin{cases}
             +1 & j\in S, \\
             -1 & j\not\in S.
             \end{cases}
\label{exhibitB}
\end{equation}
Thus each edge of $T$ in direction $i$ contributes a factor of 
$x_j$ or $x_j^{-1}$ to $x^{\dd{T}}$ for each $j$ \emph{not} equal to
$i$. The goal of this section is to show that the decoupled degree monomial
can be re-expressed as a product over the edges $e$ of $T$ of
$x_{\dir{e}}^{\pm 1}$, where the signs are determined by canonically
orienting the edges of $T$.

\begin{figure}
\begin{center}
\psfrag{empty}{$\emptyset$}
\psfrag{1}{$\{1\}$}
\psfrag{2}{$\{2\}$}
\psfrag{3}{$\{3\}$}
\psfrag{12}{$\{1,2\}$}
\psfrag{23}{$\{2,3\}$}
\psfrag{13}{$\{1,3\}$}
\psfrag{123}{$\{1,2,3\}$}
\psfrag{(a)}{(a)}
\psfrag{(b)}{(b)}
\psfrag{+1}{$+1$}
\psfrag{-1}{$-1$}
\includegraphics[scale=0.9]{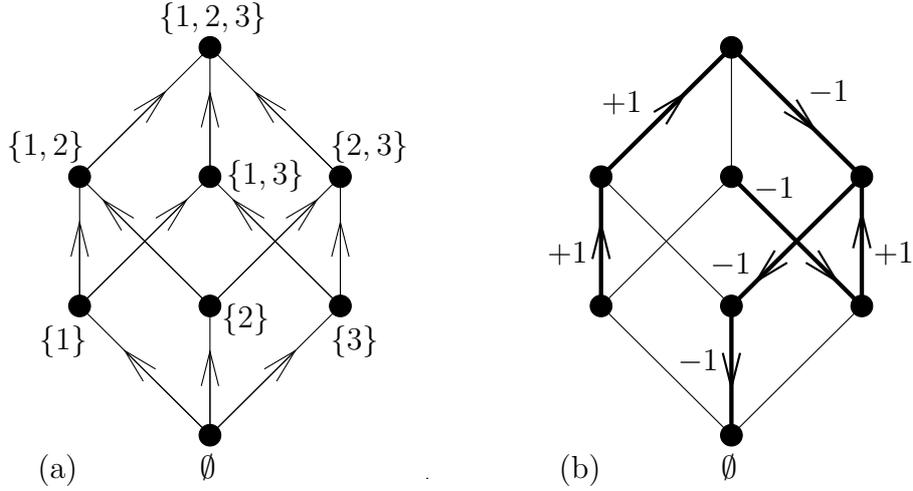}
\caption{Calculating the signs $\mu(e)$ of the edges of a spanning
  tree. (a) The oriented edges of $Q_3$. Each edge is oriented in the
  direction of increasing cardinality. (b) The oriented edges of a
  spanning tree $T$ and the resulting signs; the vertex labels have
  been omitted for clarity. Each edge of $T$ is oriented towards the
  root at $\emptyset$, and the sign is $+1$ where this orientation
  agrees with the orientation in (a), and $-1$ where it disagrees. The
  tree has weight $q_1^2q_2^3q_3^2x_1^{-1}x_2^2x_3$.}
\label{signs.fig}
\end{center}
\end{figure}

To determine the signs orient all edges of $Q_n$ ``upwards'', i.e.\ in the
direction of increasing cardinality. Root each spanning tree $T$ at 
the empty set, and orient each edge of $T$ towards the root. Each edge $e$
of $T$ now has two orientations, one from the cube and one from the tree, and
we let $\mu(e) = +1$ if the two orientations agree, and 
$\mu(e)=-1$ if the orientations disagree (see Figure~\ref{signs.fig}
for an example).
So if $\mu(e)$ is $+1$ then to get to the root $e$ must be crossed 
in the upwards direction relative to the cube,
and if $\mu(e)$ is $-1$ then $e$ must be crossed in the downwards direction
to get to the root.
With these signs we have
\begin{lemma}
\label{ddalt.lem}
\begin{equation}
x^{\dd{T}} = x_1 x_2 \cdots x_n \prod_{e\in E(T)} x_{\dir{e}}^{\mu(e)}.
\label{dd2.eq}
\end{equation}
\end{lemma}

\begin{proof}
Delete all edges of $T$ in the direction $i$. If there are $k$ of them, this 
divides $T$ into $k+1$ connected components, and $1\leq j\leq k$ of them 
will be ``upstairs'' (vertices containing $i$), and the remaining
$k+1-j$ will be downstairs. The part of $T$ upstairs has $2^{n-1}$ vertices
and Euler characteristic $j$ (since its zeroth homology has rank $j$,
and all higher homology groups are zero), so there are 
$2^{n-1}-j$ edges of $T$ upstairs, and $2^{n-1}+j-k-1$ downstairs.
By~\eqref{exhibitA} and~\eqref{exhibitB}, the degree of $x_i$ in 
$x^{\dd{T}}$ is the number of edges in $T$ upstairs minus the number of
edges downstairs, or $k+1-2j$. Now each connected component of $T$ 
upstairs must be adjacent to a unique downward edge in direction $i$,
so $j$ of the edges in direction $i$ point down and the remaining $k-j$
point up. Thus the exponent of $x_i$ in~\eqref{dd2.eq} is also
$(k-j)-j+1= k+1-2j$, so~\eqref{dd1.eq} and~\eqref{dd2.eq} agree.
\end{proof}

Note that the factors $x_1,x_2, \ldots, x_n$ in front of the product
in~\eqref{dd2.eq} are necessarily canceled by factors inside the
product, since each spanning tree must have at least one downward edge
in each direction. Thus if there are $k_i$ edges in direction $i$, then
the degree of $x_i$ in $x^{\dd{T}}$ has the opposite parity to $k_i$
and lies between $1-k_i$ and $k_i-1$.

\section{Edge slides for the three-cube}
\label{threecubeslides.sec}

In this section we define and study the edge slide operation on the 
spanning trees of $Q_3$ which we will use to prove our result. 
Goddard and Swart~\cite{goddard-swart1996} define two graphs 
$G_1$ and $G_2$ to be related by an \emph{edge move} if there
are edges $e_1\in E(G_1)$ and $e_2\in E(G_2)$ such that 
$G_2 = G_1-e_1+e_2$. Our operation may be seen as a specialisation
of this to the spanning trees of $Q_3$, in which the edges involved
in the operation are constrained by the structure of $Q_3$.
Note however that our use of the term ``edge slide'' does not agree
with that of Goddard and Swart. 

Given a graph $G$, the graph on the spanning trees of $G$ with an 
edge between $T_1$ and $T_2$ if they are related by an edge move is 
known as the \emph{tree graph} of $G$. We refer the reader to
Ozeki and Yamashita~\cite[Sec.~7.4]{ozeki-yamashita2011} for a 
survey of known results on tree graphs.

\subsection{Definition and existence}
\label{existence3.sec}

For each $i\in\{1,2,3\}$ let $\Fp{i}$ and $\Fm{i}$ be the ``upper''
and ``lower'' faces of $Q_3$ with respect to direction $i$, in other
words, the subgraphs induced by the vertices that respectively do and do not 
contain $i$. There is an obvious automorphism of $Q_3$ induced by a reflection
that exchanges \Fp{i}\ and \Fm{i}, and we will denote this automorphism
by $\sigma_i$. In terms of the symmetric difference $\symd$ this map is given
by
\[
\sigma_i(S) = S\symd\{i\}.
\]

Let $T$ be a spanning tree of $Q_3$, and let $e$ be an edge of $T$
in a direction $j\not=i$ such that $T$ does not also contain $\sigma_i(e)$.
We will say that $e$ is \emph{$i$-slidable} or \emph{slidable in
direction $i$} if deleting $e$ from $T$ and replacing it with $\sigma_i(e)$
yields a second spanning tree $T'$. We may think of this operation as
``sliding'' $e$ across a face of the cube to get a second spanning tree,
as shown in Figure~\ref{3cubeslides.fig}. We will say that an 
edge slide is ``upward'' or ``downward'' according to whether it moves
an edge from $\Fm{i}$ to $\Fp{i}$, or from $\Fp{i}$ to $\Fm{i}$.

Adding any edge of $Q_3$ to $T$ necessarily creates a cycle, so if
$e\in T$ is $i$-slidable the cycle created by adding $\sigma_i(e)$ to
$T$ must be broken by deleting $e$. This cycle must therefore contain
both $e$ and $\sigma_i(e)$, and hence at least two distinct edges in
direction~$i$. It follows that any $i$-slidable
edge must lie on the path between two edges in direction~$i$. The 
following lemma shows that a minimal path joining two edges of $T$
in direction $i$ contains exactly one $i$-slidable edge, and has as a 
consequence the fact that a spanning tree of $Q_3$ has exactly
four possible edge slides.

\begin{figure}
\begin{center}
\psfrag{empty}{$\emptyset$}
\psfrag{1}{$\{1\}$}
\psfrag{2}{$\{2\}$}
\psfrag{3}{$\{3\}$}
\psfrag{12}{$\{1,2\}$}
\psfrag{23}{$\{2,3\}$}
\psfrag{13}{$\{1,3\}$}
\psfrag{123}{$\{1,2,3\}$}
\psfrag{(a)}{(a)}
\psfrag{(b)}{(b)}
\includegraphics[scale=0.9]{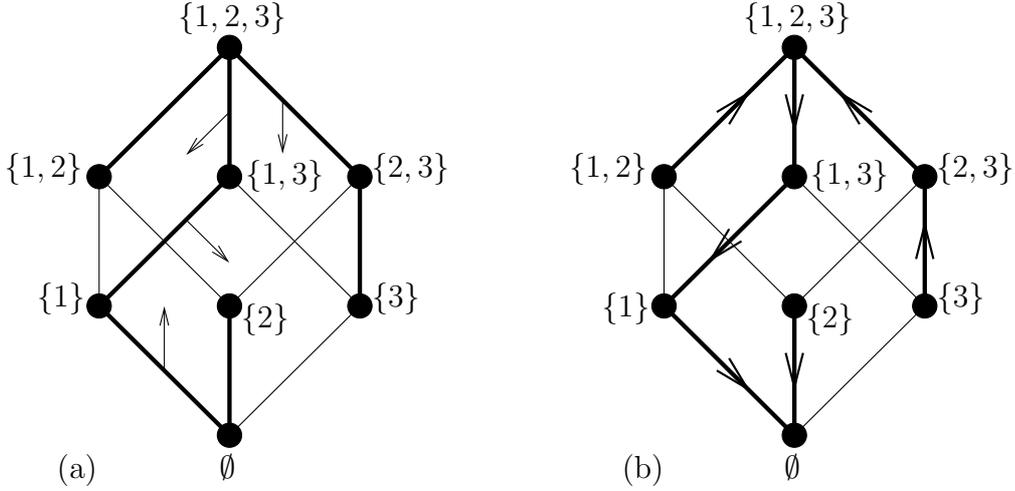}
\caption{(a) The slidable edges in a certain spanning tree $T$ of the
  three-cube.  The edge $(\emptyset,\{1\})$ may be slid up in
  direction $2$; the edge $(\{1\},\{1,3\})$ may be slid down in
  direction $1$; the edge $(\{1,3\},\{1,2,3\})$ may be slid down in
  direction $3$; and the edge $(\{2,3\},\{1,2,3\})$ may be slid down
  in direction $2$.  (b) The oriented edges of $T$ (see
  Section~\ref{orientations.sec}). The edges $(\{3\},\{2,3\})$,
  $(\{2,3\},\{1,2,3\})$ and $(\{1,2\},\{1,2,3\})$ are upward edges
  and the rest are downward.}
\label{3cubeslides.fig}
\end{center}
\end{figure}

\begin{lemma}
\label{3cubeslides.lem}
Let $T$ be a spanning tree of $Q_3$, and let $e_1$, $e_2$ be edges of
$T$ in direction $i$. If the path $P$ from $e_1$ to $e_2$ in $T$ does
not meet any other edge of $T$ in direction $i$, then $P$ contains a
unique $i$-slidable edge $e$. Moreover, if $T'$ is the result of
sliding $e$ in direction $i$, the slid edge $\sigma_i(e)$ is the
unique $i$-slidable edge on the path from $e_1$ to $e_2$ in $T'$.
\end{lemma}

\begin{remark}
In Section~\ref{higherdimensions.sec} we will see that we get
existence but not uniqueness for spanning trees of $Q_n$, $n\geq 4$.
\end{remark}

\begin{proof}
The lemma is proved by breaking it into cases
according to the length of $P$. 
The path $P$ lies in a face of $Q_3$, which is a square, and so
has length at most three; we will treat only the case where $P$ has
length exactly three, as the remaining cases are similar but easier.

When $P$ has length three it may be drawn as in
Figure~\ref{lengththree.fig}(a), in which the solid edges belong to
$T$ and $P$ is the path $(v_0,v_1,v_2,v_3)$. Consider the vertices
$\sigma_i(v_1)$ and $\sigma_i(v_2)$. They must belong to $T$; but,
since neither edge $(v_1,\sigma_i(v_1))$ nor $(v_2,\sigma_i(v_2))$
does, two of the three edges $(\sigma_i(v_0),\sigma_i(v_1))$, 
$(\sigma_i(v_1),\sigma_i(v_2))$ and $(\sigma_i(v_2),\sigma_i(v_3))$
must instead. It follows that $T$ must be given by one of the three
graphs shown in Figures~\ref{lengththree.fig}(b)--(d). In each
case, if $f=(\sigma_i(v_j),\sigma_i(v_{j+1}))$ is the dashed edge
that does not belong to $T$, then $e=\sigma_i(f)$ is the unique
$i$-slidable edge on $P$, and $f=\sigma_i(e)$ is the unique
$i$-slidable edge on the path from $e_1$ to $e_2$ in $T'$.
\end{proof}

\begin{lemma}
\label{i-slidecount.lem}
Let $T$ be a spanning tree of $Q_3$, and suppose that $T$ has
$u_i$ upward and $d_i$ downward edges in direction $i$, for a total
of $u_i+d_i=k_i$ edges in direction $i$. Then $T$ has precisely
$k_i-1$ edges that may be slid in direction $i$, and of these
$u_i$ may be slid downwards, and the remaining $d_i-1$ may be slid
upwards.
\end{lemma}

\begin{proof}
By Lemma~\ref{3cubeslides.lem} the $i$-slidable edges of $T$ must
totally disconnect the $i$-edges of $T$; since there are $k_i$ edges
in direction $i$ at least $k_i-1$ edges are required to totally
disconnect them.  So there are at least $k_i-1$ $i$-slidable edges.

To bound the number of upward and downward $i$-slides from below we
delete all edges of $T$ in direction $i$, and consider the upper and
lower faces separately. Deleting all the $i$-edges of $T$ divides
$T$ into $k_i+1$ connected components, and as seen in the proof of
Lemma~\ref{ddalt.lem}, a total of $d_i$ of these components lie
in $\Fp{i}$, with the remaining $u_i+1$ in $\Fm{i}$. Now delete
the $i$-slidable edges of $T$. Since this totally disconnects the $i$-edges
of $T$ it must further divide the $d_i$ components upstairs into
at least $k_i$ components, requiring at least $k_i-d_i=u_i$
slidable edges upstairs; and similarly it 
must further divide the $u_i+1$ components downstairs into
at least $k_i$ components, requiring at least $k_i-(u_i+1)=d_i-1$
slidable edges downstairs.

We now use the uniqueness clause of Lemma~\ref{3cubeslides.lem} 
to show that there can be no more than $k_i-1$ edges that may be
slid in direction $i$. Let $E=\{e_1,e_2,\ldots,e_{k-1}\}$ be a set
of $i$-slidable edges that totally disconnect the $i$-edges of $T$, and
let $e$ be any $i$-slidable edge of $T$. Then $e$ must lie on a path
$P$ in $T$ from one $i$-edge of $T$ to another, and we may choose $P$
so that it does not meet any other $i$-edge of $T$.  Then $e$ is
the unique $i$-slidable edge on $P$; but on the other hand, $P$ must
also cross an edge in $E$, since these totally disconnect the $i$-edges
of $T$. It follows that $e=e_j$ for some $j$, so $T$ has exactly
$k_i-1$ edges that may be slid in direction $i$. 
\end{proof}

\begin{corollary}
A spanning tree $T$ of $Q_3$ has precisely four possible edge slides.
\end{corollary}

\begin{proof}
This is immediate from Lemma~\ref{i-slidecount.lem} and the fact
that $T$ has seven edges, with at least one edge in each of the three
directions.
\end{proof}

\begin{figure}
\begin{center}
\psfrag{e1}{$e_1$}
\psfrag{e2}{$e_2$}
\psfrag{v0}{$v_0$}
\psfrag{v1}{$v_1$}
\psfrag{v2}{$v_2$}
\psfrag{v3}{$v_3$}
\psfrag{a}{(a)}
\psfrag{b}{(b)}
\psfrag{c}{(c)}
\psfrag{d}{(d)}
\psfrag{sv0}{$\sigma_i(v_0)$}
\psfrag{sv1}{$\sigma_i(v_1)$}
\psfrag{sv2}{$\sigma_i(v_2)$}
\psfrag{sv3}{$\sigma_i(v_3)$}
\includegraphics[scale=0.6]{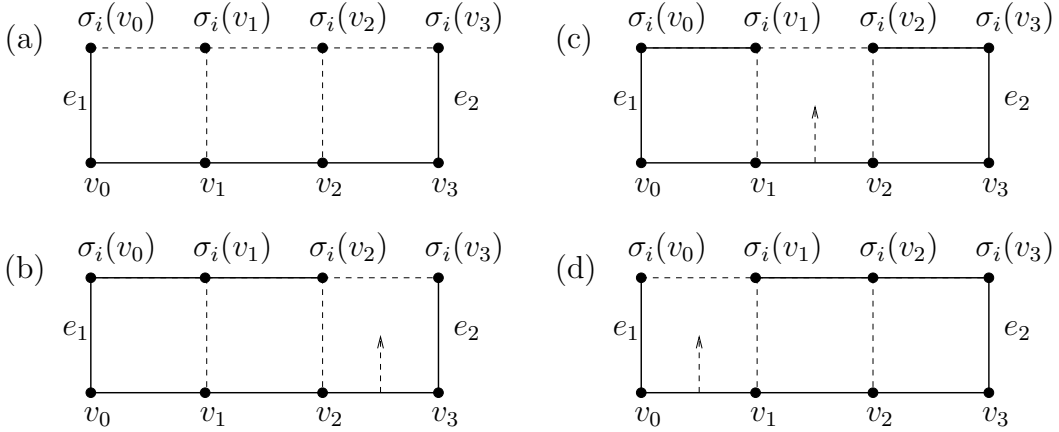}
\caption{The case where $P$ has length 3 in the proof of
  Lemma~\ref{3cubeslides.lem}.}
\label{lengththree.fig}
\end{center}
\end{figure}

\subsection{Effect on orientations}
\label{orientations.sec}

Clearly, the effect of an edge slide in direction~$i$ on the decoupled
degree monomial is to multiply it by $x_i^{\pm 2}$: by $x_i^2$ if the
edge is slid up, and $x_i^{-2}$ if the edge is slid down. Thus, the
number of upward edges in direction~$i$ must change by $\pm 1$, and
there can be no net change in the number of upward edges in other
directions. The following lemma may be used to show that in the case
of the $3$-cube this occurs through a reversal of orientation of
exactly one edge in direction~$i$.

\begin{lemma}
\label{orientations.lem}
Given an $i$-slidable edge $e$
of a spanning tree $T$, let $C$ be the cycle created by adding
$\sigma_i(e)$ to $T$; and let $T_-$ be the component of $T-e$ that
contains the root, and $T_+$ the component that does not.
Then sliding $e$ in direction $i$ reverses the orientation of
precisely those edges which belong to both $C$ and $T_+$.
The orientation of $\sigma_i(e)$ co-incides with that of $e$.
\end{lemma}

\begin{proof}
Let $T'$ be the tree resulting from the edge slide, and refer to
Figure~\ref{paths.fig}. If $u$ is a vertex belonging to $T_-$ then the
paths from $u$ to the root are the same in $T$ and $T'$, so edges in 
$T_-$ have identical orientations in $T$ and $T'$. 
If $v$ is a vertex lying in $T_+$ then the path from $v$ to the root
may be expressed in the form $PQR$, where
\begin{itemize}
\item
$P$ is the path from $v$ to the closest vertex $w$ lying on $C$;
\item
$Q$ is the path in $C$ from $w$ to $x$ that crosses $e$, where $x$ is
  the closest point on $C$ to the root;
\item
$R$ is the path from $x$ to the root.
\end{itemize}
Then the path from $v$ to the root in $T'$ has the the form $PQ'R$,
where $Q'$ is the path in $C$ from $w$ to $x$ that crosses $\sigma_i(e)$. 
It follows that edges of $T_+$ that do not lie on $C$ are unchanged in 
orientation, and by considering the cases where $v$ is the vertex of
$e$ or $\sigma_i(e)$ in $T_+$ we see that edges in $C\cap T_+$ are
reversed.
\end{proof}

\begin{figure}
\begin{center}
\psfrag{e}{$e$}
\psfrag{r}{$\emptyset$}
\psfrag{u}{$u$}
\psfrag{v}{$v$}
\psfrag{w}{$w$}
\psfrag{x}{$x$}
\psfrag{C}{$C$}
\psfrag{slide}{slide}
\psfrag{T-}{$T_-$}
\psfrag{T+}{$T_+$}
\includegraphics[scale=0.6]{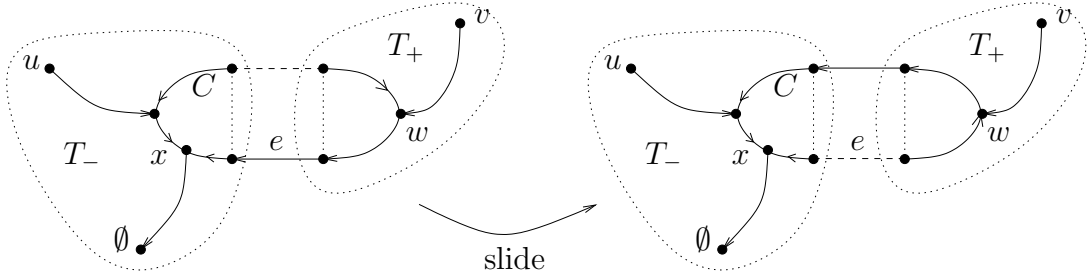}
\caption{The effect of an edge-slide on orientations. Only those edges
belonging to $C\cap T_+$ are reversed.}
\label{paths.fig}
\end{center}
\end{figure}

In the case of the $3$-cube, with notation as in
Lemma~\ref{3cubeslides.lem} it is clear that $C$ consists of 
$e$, $e_1$, $e_2$ and $P$ together with $\sigma_i(P)$. Thus, sliding
$e$ in direction $i$ reverses the orientation of exactly one edge
in direction $i$, namely whichever of $e_1$ and $e_2$ lies in $T_+$. 
Since the orientation of $e$ points from $T_+$ to $T_-$, the
edge reversed by sliding $e$ is whichever of $e_1$ and $e_2$ that
$e$ points \emph{away} from.  

\subsection{Independent slides}

We show that parallel edge-slides on a spanning tree of the $3$-cube
commute in the following sense.

Let $T$ be a spanning tree, and let $S=\{e_1,\ldots,e_k\}$ be a set of 
$i$-slidable edges of $T$. Given a vector 
$\mathbf{\eps}=(\eps_1,\ldots,\eps_k)\in\{0,1\}^k$, let $T_\mathbf{\eps}$
be the subgraph
\begin{equation}
\label{independence.eq}
T_{\mathbf{\eps}} = \bigl(T\setminus S\bigr) 
                \cup \{\sigma_i^{\eps_1}(e_1),\ldots,\sigma_i^{\eps_k}(e_k)\}.
\end{equation}
Thus, $T_\mathbf{\eps}$ is the subgraph of the cube obtained by choosing
whether or not to slide each edge $e_j\in S$ according to the
value of $\eps_j\in\{0,1\}$.
We will say $S$ is an
\emph{independently slidable set} if $T_\mathbf{\eps}$ is a spanning tree
of the cube for all $\mathbf{\eps}\in\{0,1\}^k$.

In Lemma~\ref{3cubeindependence.lem} we show that parallel edge
slides in a spanning tree of the $3$-cube are independent in the above
sense, which will allow us to count the trees combinatorially.

\begin{lemma}
\label{3cubeindependence.lem}
The set $S_i(T)$
consisting of the $i$-slidable edges of a spanning tree $T$ of the
$3$-cube is an independently slidable set. 
\end{lemma}

\begin{proof}
Referring again to Figure~\ref{lengththree.fig}, we see that $i$-slidability
of a given edge $e\in S_i(T)$ is a local property, depending only on the 
minimal cycle $C$ obtained by adding $\sigma_i(e)$ to $T$. The edges
of $C$ are unaffected by $i$-slides of other edges in $S_i(T)$, and so
$e$ remains slidable regardless of how these other edges are slid.
\end{proof}

\section{Counting the spanning trees of the three-cube}
\label{thecount.sec}

\subsection{Strategy}

We will now use the results of the previous section to show
that $Q_3$ has $2^4\cdot 2^3\cdot 3=384$ spanning trees. We will do 
this by constructing a projection from $\tree{Q_3}$ onto a space of trees
that are easily counted; showing that there are $2^3\cdot3$ trees
in this family; and that each fibre of the projection has size $2^4$. 

For the target of the projection we define a spanning tree of $Q_3$ to
be \emph{upright} if all of its edges are oriented downwards. An
example appears in Figure~\ref{upright.fig}. We denote the set of
upright trees of $Q_3$ by $\utree{Q_3}$, and will count the trees
by constructing a projection $\pi\co\tree{Q_3}\to\utree{Q_3}$. 

\begin{figure}
\begin{center}
\psfrag{empty}{$\emptyset$}
\psfrag{1}{$\{1\}$}
\psfrag{2}{$\{2\}$}
\psfrag{3}{$\{3\}$}
\psfrag{12}{$\{1,2\}$}
\psfrag{23}{$\{2,3\}$}
\psfrag{13}{$\{1,3\}$}
\psfrag{123}{$\{1,2,3\}$}
\includegraphics[scale=0.8]{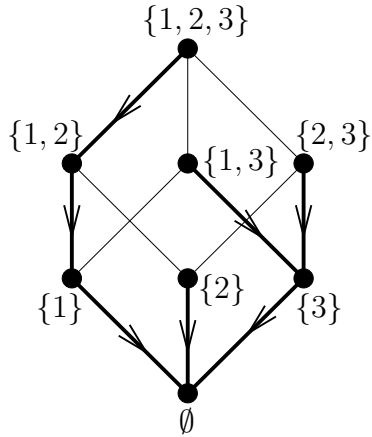}
\caption{An upright spanning tree of $Q_3$.  The associated
section $\phi$ (see Section~\ref{upright.sec}) is determined
by $\phi(\{1,2,3\})=3$, $\phi(\{1,2\})=2$, $\phi(\{1,3\})=1$ and 
$\phi(\{2,3\})=2$.}
\label{upright.fig}
\end{center}
\end{figure}

If $A$ is a subset of $X$ we will say that a function $f\co X\to A$ is
a \emph{retraction} if the restriction of $f$ to $A$ is the identity.
We will build the projection $\pi$ up as a composition of retractions, 
and to this end we define a spanning tree to be \emph{upright with
respect to direction $i$} if all of its edges in direction $i$ are
oriented downwards. We write $\downwardi{Q_3}$ for the set of trees
that are upright with respect to direction $i$, and observe that
\[
\utree{Q_3}=\bigcap_{i=1}^3 \downwardi{Q_3}.
\]
We will proceed by constructing retractions 
$\pi_i\co\tree{Q_3}\to\downwardi{Q_3}$, and then set
$\pi=\pi_1\circ\pi_2\circ\pi_3$. 

\subsection{The retractions}
\label{retract.sec}

Given a spanning tree $T$ of $Q_3$ we set 
$S=S_i(T)=\{e_1,\ldots,e_{k_i-1}\}$ in
equation~\eqref{independence.eq}, and let
\[
\islide{T}=\left\{T_\eps|\eps\in\{0,1\}^{|S_i(T)|}\right\}. 
\]
By Lemma~\ref{3cubeindependence.lem} this set consists of $2^{k_i-1}$
spanning trees of $Q_3$, and we claim that these sets form a partition
of $\tree{Q_3}$.  Indeed, suppose that $T'=T_\eps\in\islide{T}$.  Then
$T_\eps$ has the same number of $i$-edges as $T$, and hence the same
number of $i$-slidable edges as $T$, by Lemma~\ref{i-slidecount.lem};
thus we necessarily have $S_i(T_\eps)=\{\sigma_i^{\eps_1}(e_1),\ldots,
\sigma_i^{\eps_{k_i-1}}(e_{k_i-1})\}$, since these edges are
$i$-slidable in $T_\eps$. Then for $\delta\in\{0,1\}^{|S_i(T_\eps)|}$
we have $(T_\eps)_\delta=T_{\eps+\delta}$, and so 
$\islide{T_\eps}=\islide{T}$. The sets $\islide{T}$ are therefore
disjoint or equal, and since $T$ necessarily belongs to $\islide{T}$,
they form a partition as claimed. We note in passing that these sets
are the equivalence classes of the relation 
$\sim_i$ defined by
$T_1\sim_i T_2$ if and only if $T_1$ can be transformed into $T_2$ by
a series of $i$-slides.

Now, for each spanning tree $T$ of $Q_3$, carrying out all possible downward
$i$-slides gives a unique choice of $T_{\tilde\eps}\in\islide{T}$ such
that $\sigma_i^{\tilde\eps_j}(e_j)$ lies in $\Fm{i}$ for each
$j=1,\ldots,k_i-1$. For such a tree only upward $i$-slides are
possible, and so $T_{\tilde\eps}$ can have only downward $i$-edges, by
Lemma~\ref{i-slidecount.lem}. Setting $\pi_i(T)=T_{\tilde\eps}$ we 
therefore obtain a retraction $\pi_i\co\tree{Q_3}\to\downwardi{Q_3}$. 
For each $T\in\downwardi{Q_3}$ the fibre of this map is
$\islide{T}$, and so has cardinality $2^{k_i-1}$, where $k_i$ is the
number of $i$-edges of $T$. 

We now consider the composition $\pi=\pi_1\circ\pi_2\circ\pi_3$.
For a spanning tree $T$ of $Q_3$ the tree $\pi(T)$ is obtained by
\begin{enumerate}
\item
carrying out all possible downward edge slides in direction 3 in
$T$ to get $\pi_3(T)$; then
\item
carrying out all possible downward edge slides in direction 2
in $\pi_3(T)$, to get $\pi_2(\pi_3(T))$; then
\item
carrying out all possible downward edge slides in direction 1 in
$\pi_2(\pi_3(T))$, to get $\pi_1(\pi_2(\pi_3(T)))=\pi(T)$.
\end{enumerate}
An example appears in Figure~\ref{retraction.fig}. For this map
we claim
\begin{lemma}
\label{preimage.lem}
The map $\pi=\pi_1\circ\pi_2\circ\pi_3$ is a retraction from
$\tree{Q_3}$ onto $\utree{Q_3}$. For each tree $T\in\utree{Q_3}$
the preimage of $T$ contains exactly $2^4$ trees.
\end{lemma}

\begin{figure}
\begin{center}
\psfrag{empty}{$\emptyset$}
\psfrag{1}{$\{1\}$}
\psfrag{2}{$\{2\}$}
\psfrag{3}{$\{3\}$}
\psfrag{12}{$\{1,2\}$}
\psfrag{23}{$\{2,3\}$}
\psfrag{13}{$\{1,3\}$}
\psfrag{123}{$\{1,2,3\}$}
\psfrag{p1}{$\pi_1$}
\psfrag{p2}{$\pi_2$}
\psfrag{p3}{$\pi_3$}
\includegraphics[scale=0.62]{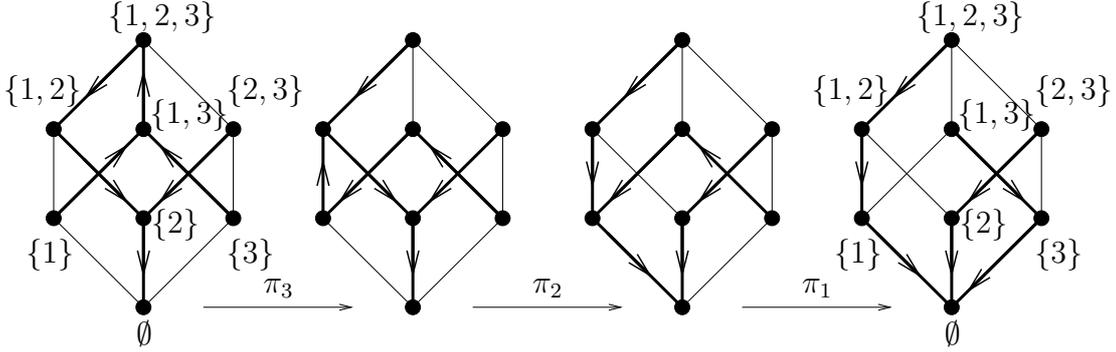}
\caption{An example of the retraction $\pi=\pi_1\circ\pi_2\circ\pi_3$.
By successively carrying out all possible downward edges slides in
directions 3, 2, 1 in turn we arrive at the upright tree on the right.}
\label{retraction.fig}
\end{center}
\end{figure}

\begin{proof}
Recall from Section~\ref{orientations.sec} that an $i$-slide has no
net effect on the number of upward edges in directions other than $i$.
This implies that 
\begin{equation}
\label{nestingretract.eq}
\pi_i(\downward{j}{Q_3})\subseteq\downward{j}{Q_3}
\quad \text{for all $i$ and $j$,}
\end{equation}
 and that
\begin{equation}
\label{pullback.eq}
\pi_i^{-1}(\downward{j}{Q_3})\subseteq\downward{j}{Q_3}
\quad\text{ for $j\neq i$}.
\end{equation}
The relation
\[
\pi(\tree{Q_3})=\pi_1\circ\pi_2\circ\pi_3(\tree{Q_3})
       \subseteq\downward{1}{Q_3}\cap\downward{2}{Q_3}\cap\downward{3}{Q_3}
              = \utree{Q_3}
\]
now follows immediately from equation~\eqref{nestingretract.eq}
and the fact that $\pi_i(\tree{Q_3})=\downward{i}{Q_3}$. Moreover if
$T$ is upright then $\pi_i(T)=T$ for all $i$, so $\pi(T)=T$ and $\pi$
is a retraction.

Suppose now that $T$ is an upright tree with $k_i$ edges in direction
$i$ for $i=1,2,3$. Then the preimage of $T$ under $\pi_1$ is
$\islide[1]{T}$, which consists of $2^{k_1-1}$ trees, each with the
same number of edges in each direction as $T$, and each lying in
$\downward{2}{Q_3}\cap\downward{3}{Q_3}$, by
equation~\eqref{pullback.eq}. The preimage of each tree
$T_\eps\in\islide[1]{T}$ under $\pi_2$ is then $\islide[2]{T_\eps}$,
which consists of
  $2^{k_2-1}$ spanning trees, each of which must belong to
  $\downward{3}{Q_3}$ and have $k_3$ edges in direction $3$, and
  pulling each such tree $(T_\eps)_\delta$ back under $\pi_3$ we get 
$\islide[3]{(T_\eps)_\delta}$, which consists of
$2^{k_3-1}$ spanning
  trees. We therefore get a total of
\[
2^{k_1-1}\times2^{k_2-1}\times2^{k_3-1}= 2^{k_1+k_2+k_3-3}=2^4
\]
trees in $\pi^{-1}(T)$. 
\end{proof}

\subsection{The number of upright trees}
\label{upright.sec}

Given an upright spanning tree $T$, the first edge on the path from a
non-root vertex $S$ to the root must be in a direction $\phi_T(S)=i$
belonging to $S$. To each upright tree we may therefore associate a
function $\phi_T\co\power{[3]}\setminus\{\emptyset\}\to[3]$ such that
$\phi_T(S)\in S$ for all $S$. We will say that such a function is a
\emph{section} of $\powernon{3}=\power{[3]}\setminus\{\emptyset\}$.

Conversely, given a section $\phi$ of $\powernon{3}$, let $T(\phi)$ be
the subgraph of $Q_3$ containing all eight vertices and the seven edges
\[
\Bigl\{\bigl\{S,S\setminus\{\phi(S)\}\bigr\}:S\in\powernon{3}\Bigr\}.
\]
At each vertex $S$ of $Q_3$ other than $\emptyset$ there is a unique
edge in $T(\phi)$ connecting $S$ to a vertex of cardinality $|S|-1$,
and following such edges one obtains a path in $T(\phi)$ from $S$ to
the root $\emptyset$.  It follows that $T(\phi)$ is connected, and
since it has seven edges and includes all eight vertices it must be a
spanning tree. Moreover, these paths show that $T(\phi)$ is upright,
and it is easily seen that it has associated section $\phi$.

This all proves the following lemma, and then Corollary~\ref{count.cor}
follows immediately from Lemmas~\ref{preimage.lem} and~\ref{uprightcount.lem}.

\begin{lemma}
\label{uprightcount.lem}
The upright spanning trees of $Q_3$ are in bijection with the sections
of $\powernon{3}$. Consequently there are $2^3\times 3=24$ upright trees.
\end{lemma}

\begin{corollary}
\label{count.cor}
The $3$-cube has $2^4\cdot2^3\cdot 3=384$ spanning trees. 
\end{corollary}

\subsection{A bijective count}
\label{bijection.sec}

With some additional bookkeeping we may establish a bijection $\Phi$
between $\tree{Q_3}$ and the set of \emph{signed sections} of
$\ptwo{3} = \{S\in\power{[3]}:|S|\geq 2\}$. These are functions
$\phi=\phi_d\times\phi_s\co\ptwo{3}\to[3]\times\{\pm1\}$ such that
$\phi_d(S)\in S$ for all $S$. Given such a function $\phi$ we may
define its weight to be
\[
q^{\dir{\phi}}x^{\sgn{\phi}}
  = q_1q_2q_3\prod_{S\in\ptwo{3}} q_{\phi_d(S)}x_{\phi_d(S)}^{\phi_s(S)},
\]
and the bijection will be weight preserving in the sense that
$q^{\dir{T}}x^{\dd{T}}$ will equal the weight of the associated 
section $\phi$. This gives a bijective proof of the $n=3$ case
of the Martin-Reiner formula~\eqref{weightedcount.eq}.

Given a spanning tree $T$ of $Q_3$, the upright tree $\pi(T)$ has a
canonical associated section $\phi_{\pi(T)}$ of
$\powernon{3}$. Restricting $\phi_{\pi(T)}$ to the sets of size two or
more gives an (unsigned) section $\phi$ of $\ptwo{3}$, and this restriction
completely determines $\phi_{\pi(T)}$ and hence $\pi(T)$. 
Moreover $q^{\dir{T}}=q^{\dir{\pi(T)}}=q^{\dir{\phi}}$. We may
therefore define $\Phi=\Phi_d\times\Phi_s$ so that 
\[
\Phi_d(T) = \phi_{\pi(T)}\Bigr|_{\ptwo{3}}.
\]
It remains to define the signs, and we will do this by studying the
way in which the edge slides taking $T$ to $\pi(T)$ affect the
orientations of the edges.

In Section~\ref{orientations.sec} we saw that sliding an edge $e$ of
$T$ in direction $i$ reverses the orientation of exactly one edge in
direction $i$. We claim that the reversed $i$-edge is the same for
all trees $T_\eps\in\islide{T}$:

\begin{lemma}
\label{reversededge.lem}
Let $e$ be an $i$-slidable edge of a spanning tree $T$ of $Q_3$, and let
$f$ be the $i$-edge of $T$ whose orientation is reversed by sliding $e$.
Then, for any tree $T_\eps\in\islide{T}$, $f$ is also the $i$-edge
reversed by sliding whichever
of $e,\sigma_i(e)$ belongs to $T_\eps$.
\end{lemma}

\begin{proof}
Let $C_e$ be the cycle formed by adding $\sigma_i(e)$ to
$T$. Then as noted at the end of Section~\ref{orientations.sec}, $C_e$
consists of two $i$-edges $e_1$ and $e_2$, and the path $P$ joining
them in $T$ together with $\sigma_i(P)$, and the edge $f$ 
reversed by sliding $e$ is whichever of $e_1$ and $e_2$ that $e$ points away
from. Neither $e$ nor $\sigma_i(e)$ lies on the corresponding cycle
$C_{e'}$ for any other $i$-slidable edge $e'\in S_i(T)$, and therefore
the orientation of $e$ or $\sigma_i(e)$ cannot be affected by sliding
$e'$ or $\sigma_i(e')$, by Lemma~\ref{orientations.lem}. It follows
that $e$ and $\sigma_i(e)$ always point away from the same edge
$e_1$ or $e_2$ regardless of how the other edges in $S_i(T)$ are 
slid, proving the lemma.
\end{proof} 

To define $\Phi_s(T)$ at vertices where $\Phi_d(T)(S)=i$ we
write $\pi=\pi_1\circ\pi_2\circ\pi_3$ in the form
$\alpha\circ\pi_i\circ\beta$, where $\beta$ is the composition of
the retractions preceding $\pi_i$, and $\alpha$ the composition 
of those that follow. (Thus for example when $i=3$, $\beta$ is
the identity and $\alpha=\pi_1\circ\pi_2$.) The tree 
$\pi_i\circ\beta(T)$ is obtained from $\beta(T)$ by carrying
out all possible downward $i$-slides, and we partition the 
$i$-edges of $\beta(T)$ into three sets $P_i$, $N_i$ and $Z_i$:
\begin{itemize}
\item
$P_i$ consists of the $i$-edges of $\beta(T)$ that may be reversed by
downward $i$-slides in $\beta(T)$ (and hence the $i$-edges that are
reversed in obtaining $\pi_i\circ\beta(T)$ from $\beta(T)$);
\item
$N_i$ consists of the $i$-edges of $\beta(T)$ that may be reversed
by upward $i$-slides in $\beta(T)$; and
\item
$Z_i$ consists of the unique $i$-edge of $\beta(T)$ that may not
be reversed by an $i$-slide. This is the $i$-edge that belongs to the
component containing the root when the $i$-slidable edges 
of $\beta(T)$ are deleted.
\end{itemize}
We note that this partition uniquely determines
$\beta(T)$ from $\pi_i\circ\beta(T)$, by Lemma~\ref{reversededge.lem}.

We now keep track of this partition as $\pi_i\circ\beta(T)$ is
transformed into $\pi(T)$ by $\alpha$. During this transformation
the $i$-edges may be slid in directions $j\neq i$, but we keep
track of them through this movement to obtain a corresponding
partition $\{P_i,N_i,Z_i\}$ of the $i$-edges of $\pi(T)$. 
There are now two possibilities:
\begin{enumerate}
\item
$Z_i$ consists of the edge $\{\emptyset,\{i\}\}$. In this case
we simply set $\Phi(T)(S) = (i,+1)$ if the first edge on the path
in $\pi(T)$ to the root is an $i$-edge belonging to $P_i$, and
$\Phi(T)(S) = (i,-1)$ if the first edge on the path
in $\pi(T)$ to the root is an $i$-edge belonging to $N_i$. 
\item
\label{swap.case}
$Z_i=\{e\}$, for some $i$-edge $e\neq\{\emptyset,\{i\}\}$. In this
case we modify the partition by swapping $e$ and $\{\emptyset,\{i\}\}$,
and then assign signs as in the previous case.
\end{enumerate}

\begin{example}
We determine the signs associated with direction 3 for the tree
appearing on Figure~\ref{retraction.fig}. In the leftmost tree the
3-slidable edges are $\{\{2\},\{1,2\}\}$, which may be slid up to
reverse the orientation of the $3$-edge $\{\{1,2\},\{1,2,3\}\}$, and
$\{\{1,3\},\{1,2,3\}\}$, which may be slid down to reverse
$\{\{1\},\{1,3\}\}$. So $P_3$ contains the edge $\{\{1\},\{1,3\}\}$
only, $N_3$ contains the edge $\{\{1,2\},\{1,2,3\}\}$ only, and $Z_3$
contains the remaining $3$-edge $\{\{2\},\{2,3\}\}$, which is not
reversed by either $3$-slide. Under the $2$- and $1$-slides the edge
$\{\{1\},\{1,3\}\}$ belonging to $P_3$ is moved to
$\{\emptyset,\{3\}\}$, putting us in case~\eqref{swap.case} above.
The $3$-edge belonging to $Z_3$ is still $\{\{2\},\{2,3\}\}$, so
we assign the plus sign to the vertex $\{2,3\}$. The associated signed
section in full is
\begin{align*}
\Phi(\{1,2\}) &= (2,+1), & \Phi(\{1,3\}) &= (1,+1) \\
\Phi(\{2,3\}) &= (3,+1), & \Phi(\{1,2,3\}) &= (3,-1).
\end{align*}
\end{example}

It is clear by construction that the resulting map $\Phi$ is weight
preserving, so it remains to check that $T$ is in fact determined by
the associated signed section $\Phi(T)$. The upright tree $\pi(T)$ may
be recovered from $\Phi_d(T)$, and we partition the $i$-edges of
$\pi(T)$ as $\{P_i,N_i,Z_i\}$ according to the signs.  We may then
carry out slides in directions 1, 2 and 3 in turn so as to reverse the
orientations of the edges belonging to $P_i$ for each $i$,
keeping track of each partition as we do so. The only
difficulty that can arise is if at the $i$th stage the $i$-edge that
cannot be reversed belongs to $P_i$ (this can only occur for $i\geq
2$). In that case we reverse the edge belonging to $Z_i$ in its
place, which has the effect of undoing the modification made to the
partition in case~\eqref{swap.case} above.

\subsection{The edge slide graph of the three cube}
\label{edgeslidegraph.sec}

We define the \emph{edge slide graph of} $Q_3$ to be the 
graph \eslide{3}\ with vertices the spanning trees of $Q_3$, and an
edge between trees $T_1$ and $T_2$ if they are related by a single
edge slide. The edge slide graph is a subgraph of the 
tree graph~\cite[Sec.~7.4]{ozeki-yamashita2011}
of $Q_3$, which has an edge between $T_1$ and $T_2$
if $|E(T_1)\setminus E(T_2)|=1$.
We conclude this section by describing the
connected components of \eslide{3}.

Edge slides do not change the direction monomial
$q^{\dir{T}}=q_1^{k_1}q_2^{k_2}q_3^{k_3}$, so if 
$q^{\dir{T_1}}\neq q^{\dir{T_2}}$ then $T_1$ and $T_2$ lie in different
components. Thus, it suffices to understand the subgraphs 
$\eslidesig{k_1,k_2,k_3}$ consisting of the trees with direction
monomial $q_1^{k_1}q_2^{k_2}q_3^{k_3}$. Moreover, any permutation of
$\{1,2,3\}$ induces an automorphism of $Q_3$, and hence of \eslide{3}, so
we may consider the triple $(k_1,k_2,k_3)$ up to permutation.
We will refer to this triple as the \emph{signature} of $T$, and up
to permutation we find that there are three possible signatures, namely
$(4,2,1)$, $(3,3,1)$ and $(3,2,2)$. 

A spanning tree with $k_3=1$ consists of two spanning trees of $Q_2$,
lying in $\Fm{3}$ and $\Fp{3}$, joined by an edge in direction 3. No
edge slide in direction 3 is possible, while the spanning trees of
$Q_2$ in $\Fm{3}$ and $\Fp{3}$ each have a single possible edge slide,
and the edge joining them may be slid in either direction
1 or 2. It is easily seen that these four slides may be made
independently, so each such tree belongs to a component of
$\eslide{3}$ isomorphic to $Q_4$. We get one such component for
signature $(4,2,1)$, and two for $(3,3,1)$, characterised by the
signature $(k_1',k_2')$ of the spanning tree of $\Fm{3}$, which
is invariant under edge slides in directions 1 and 2.

The subgraph \eslidesig{3,2,2}\ is more interesting, as now
edge slides in all three directions are possible. We find that there
are four possible upright trees with signature $(3,2,2)$, occurring
in two mirror image pairs, and that any two of these may be 
connected by a series of edge slides. Since any spanning tree
may also be connected to an upright tree by a series of edge slides
this implies that \eslidesig{3,2,2}\ is connected. This gives us
a connected component with 64 vertices, consisting of the 
$16\times 4$ spanning trees associated with these four upright 
trees by $\pi$. 

In total \eslide{3}\ has $6\times1+3\times2=12$ components isomorphic
to $Q_4$, and three 64-vertex components isomorphic to \eslidesig{3,2,2},
accounting for all $24\times 16=384$ spanning trees of $Q_3$.  
The structure of the 64-vertex component has been found by Lyndal
Henden~\cite{henden2011} as an undergraduate summer research project.

\section{Edge slides in higher dimensions}
\label{higherdimensions.sec}

The definition of an edge slide in Section~\ref{existence3.sec} was
stated only for a spanning tree of the three-cube, but it applies just
as well to a spanning tree of $Q_n$ for any $n\geq 2$. In this section
we prove that a spanning tree with $k_i$ edges in direction $i$ always
has at least $k_i-1$ edges that may be slid in direction $i$, but show
by example that the methods of this paper do not readily extend to
count the spanning trees of $Q_n$ for $n\geq 4$.

\subsection{Existence}
\label{existence-n.sec}

As in the three-dimensional case, an $i$-slidable edge must lie on the
path between two edges in direction~$i$. We prove the following 
existence theorem, and deduce three corollaries.

\begin{theorem}
\label{ncubeslides.th}
Let $T$ be a spanning tree of $Q_n$, and let $e_1$ and $e_2$ be
edges of $T$ in direction $i$. Then there is an $i$-slidable edge on
the path from $e_1$ to $e_2$ in $T$. 
\end{theorem}

\begin{proof}
Without loss of generality we may assume that the path $P$ from
$e_1$ to $e_2$ in $P$ does not meet any other edge of $T$ in direction
$i$. Let $P=(v_0,v_1,\ldots,v_m)$, where the vertices $v_0$ and $v_m$ are 
incident with $e_1$ and $e_2$ respectively. 

For each vertex $v_j$ of $P$ let $\phi(v_j)$ be the first vertex of
$P$ on the path from $\sigma_i(v_j)$ to $v_j$ in $T$. 
Clearly, 
$\phi(v_0)=v_0$, and $\phi(v_m)=v_m$. 
Suppose that there is an edge $f=(v_\ell,v_{\ell+1})$ of $P$ such that
$\phi(v_\ell)$ and $\phi(v_{\ell+1})$ lie on \emph{opposite} sides
of $f$. Then adding $\sigma_i(f)$ to $T$ creates a cycle that is broken
by deleting $f$, so $f$ is $i$-slidable. 

If there is no such edge $f$ then we may show by induction that
we have $\phi(v_j)\in\{v_0,v_1,\ldots,v_{j-1}\}$ for $1\leq j\leq m$. But this
contradicts the fact that $\phi(v_m)=v_m$, so there must be an
$i$-slidable edge on $P$. 
\end{proof}

Using Theorem~\ref{ncubeslides.th} and ideas that have appeared in
Section~\ref{threecubeslides.sec} we may easily deduce the following
corollaries. 

\begin{corollary}
\label{countofslides.cor}
Let $T$ be a spanning tree of $Q_n$, and suppose that $T$ has
$u_i$ upward and $d_i$ downward edges in direction $i$, for a total
of $u_i+d_i=k_i$ edges in direction $i$. Then $T$ has at least
$k_i-1$ edges that may be slid in direction $i$, and of these
at least $u_i$ may be slid downwards, and at least $d_i-1$ may be slid
upwards.
\end{corollary}

\begin{corollary}
A spanning tree of $Q_n$ has at least $2^n-n-1$ possible edge slides.
\end{corollary}

\begin{corollary}
Given a spanning tree $T$ of $Q_n$ there is a sequence of downward
edge slides that transforms $T$ into an upright tree.
\end{corollary}

We note further that the argument of
Section~\ref{upright.sec} may be used to show that the upright trees
of $Q_n$ are in bijection with the sections of $\powernon{n}$, so that
$Q_n$ has a total of $\prod_{k=1}^n k^{\binom{n}{k}}$
upright trees.

\subsection{Counterexamples}
\label{counterexamples.sec}

\begin{figure}
\begin{center}
\psfrag{e}{$e$}
\includegraphics[scale=0.55]{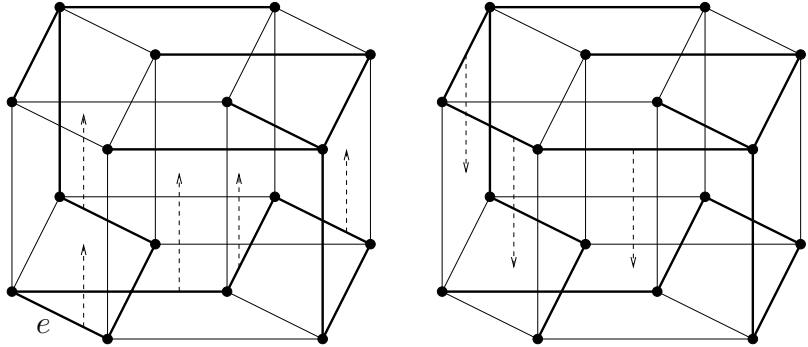}
\caption{``Extra'' edge slides when $n=4$. The tree on the left has
only two vertical edges, but there are five edges (indicated with 
dashed arrows) on the path joining them that may be slid 
vertically. 
The tree on the right shows the result of sliding
the edge labelled $e$ at left. This now has three vertically
slidable edges on the path joining the two vertical edges.}
\label{neq4example.fig}
\end{center}
\end{figure}

Our construction of the retraction $\pi$ from $\tree{Q_3}$ to
$\utree{Q_3}$ depended on the fact that a spanning tree of $Q_3$ with
$k_i$ edges in direction $i$ has precisely $k_i-1$ $i$-slidable edges,
which may all be slid independently (in the sense that sliding any one
of them has no effect on the slidability of the others).  In
Figures~\ref{neq4example.fig} and~\ref{neq5example.fig} we show by
example that this fails for $n\geq 4$. Consequently, the methods of
this paper do not readily extend to count the spanning trees of $Q_n$
for $n\geq 4$.

The tree on the left in Figure~\ref{neq4example.fig} has only two 
vertical edges, but five vertically slidable edges on the 
path joining them. When any one of these five edges is slid vertically 
the other four necessarily cease to be slidable, because the vertical
edges are now disconnected downstairs. However, if the
edge labelled $e$ is slid two vertically slidable edges are created
upstairs, as seen in the tree on the right.

Figure~\ref{neq5example.fig} shows that even when a spanning
tree has precisely $k_i-1$ $i$-slidable edges, the $i$-slides
cannot necessarily be made independently. The tree on the left has
three edges joining the upper and lower $4$-cubes, and precisely 
two edges that may be slid from one $4$-cube to the other. However,
if both are slid the result is not a tree, as seen in the graph
on the right.

\begin{figure}
\begin{center}
\includegraphics[scale=0.5]{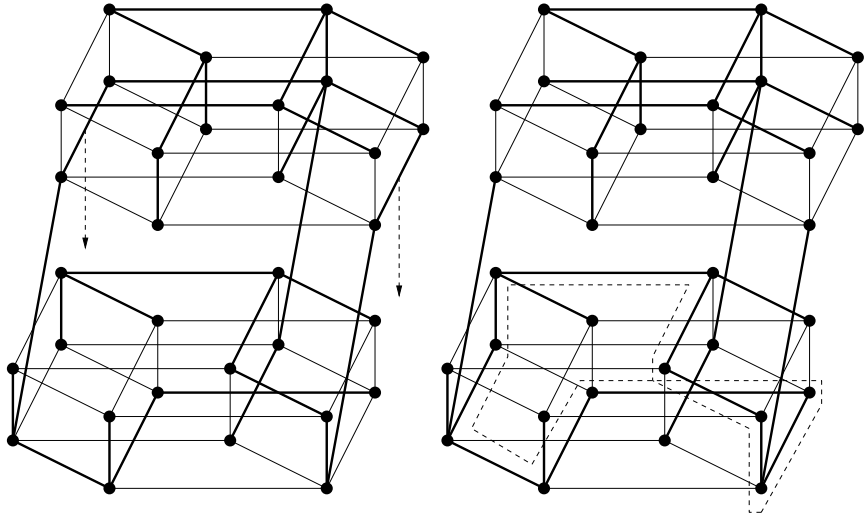}
\caption{Parallel edge slides need not be independent, even when there
are no more edge slides in that direction than expected. The figure
on the left illustrates a spanning tree of the $5$-cube, with some
edges of the $5$-cube omitted for clarity. The tree has three
edges joining the upper and lower $4$-cubes, and precisely 
two edges that may be slid from one $4$-cube to the other.
The figure on the right
shows the result of sliding both. This contains a cycle, indicated by the
dashed line which shadows it.}
\label{neq5example.fig}
\end{center}
\end{figure}

\bibliographystyle{plain}
\bibliography{cubetrees}

\begin{thebibliography}{1}

\bibitem{bernardi2012}
O.~Bernardi.
\newblock On the spanning trees of the hypercube and other products of graphs.
\newblock E-print arXiv:1207.0896v3 [math.CO], 2012.

\bibitem{goddard-swart1996}
W.~Goddard and H.C. Swart.
\newblock Distances between graphs under edge operations.
\newblock {\em Discrete Math.}, 161(1-3):121--132, 1996.

\bibitem{henden2011}
L.~Henden.
\newblock The edge slide graph of the 3-cube.
\newblock {\em Rose-Hulman Undergraduate Mathematics Journal}, 12(2), 2011.
\newblock \url{http://www.rose-hulman.edu/mathjournal/v12n2.php}.

\bibitem{martin-reiner}
J.L.\ Martin and V.\ Reiner.
\newblock Factorizations of some weighted spanning tree enumerators.
\newblock {\em J.\ Comb.\ Theory Ser.\ A}, 104(2):287--300, 2003.

\bibitem{moon67}
J.W.\ Moon.
\newblock Various proofs of {C}ayley's formula for counting trees.
\newblock In {\em A seminar on Graph Theory}, pages 70--78. Holt, Rinehart and
  Winston, 1967.

\bibitem{ozeki-yamashita2011}
K.~Ozeki and T.~Yamashita.
\newblock Spanning trees: a survey.
\newblock {\em Graphs Combin.}, 27(1):1--26, 2011.

\bibitem{prufer18}
H.\ Pr{\"u}fer.
\newblock Neuer {B}eweis eines {S}atzes {\"{u}}ber {P}ermutationen.
\newblock {\em Arch.\ Math.\ Phys.}, 27:742--744, 1918.

\bibitem{stanley-II}
R.P.\ Stanley.
\newblock {\em Enumerative Combinatorics, Volume II}.
\newblock Number~62 in Cambridge Studies in Advanced Mathematics. Cambridge
  University Press, 1999.

\end{thebibliography}

\end{document}